\documentclass[]{interact}

\usepackage{epstopdf}
\usepackage[caption=false]{subfig}
\usepackage{hyperref}
\usepackage{algcompatible}
\usepackage{amsmath,amssymb,latexsym,amscd,amsxtra,graphicx,graphpap,amsthm}
\usepackage{mathtools}
\usepackage{newtxtext,newtxmath}
\usepackage[justification=centering, labelfont=bf]{caption}
\usepackage[ruled, linesnumbered, noend]{algorithm2e}
\usepackage[numbers,sort&compress]{natbib}

\bibpunct[, ]{[}{]}{,}{n}{,}{,}
\renewcommand\bibfont{\fontsize{10}{12}\selectfont}

\theoremstyle{plain}
\newtheorem{theorem}{Theorem}[section]
\newtheorem{lemma}[theorem]{Lemma}

\newtheorem{proposition}[theorem]{Proposition}

\theoremstyle{definition}
\newtheorem{definition}[theorem]{Definition}
\newtheorem{example}[theorem]{Example}
\usepackage{graphicx}
\usepackage{amsmath}
\usepackage{amssymb}
\usepackage{booktabs}
\theoremstyle{remark}
\newtheorem{remark}{Remark}

\begin{document}

\articletype{ARTICLE TEMPLATE}

\title{A neurodynamic approach for a class of pseudoconvex semivectorial bilevel optimization problems}

\author{
\name{Tran Ngoc Thang\textsuperscript{a}\thanks{CONTACT T.N.Thang. Email: thang.tranngoc@hust.edu.vn} and Dao Minh Hoang\textsuperscript{b} and Nguyen Viet Dung \textsuperscript{a}}
\affil{\textsuperscript{a}School of Applied Mathematics and Informatics, Hanoi University of Science and Technology,\\ Hanoi, Vietnam } 
\affil{\textsuperscript{b} ISIMA, Clermont Auvergne University, Clermont-Ferrand, France} 
}

\maketitle

\begin{abstract}
The article proposes an exact approach to find the global solution of a nonconvex semivectorial bilevel optimization problem, where the objective functions at each level are pseudoconvex, and the constraints are quasiconvex. Due to its non-convexity, this problem is challenging, but it attracts more and more interest because of its practical applications. The algorithm is developed based on monotonic optimization combined with a recent neurodynamic approach, where the solution set of the lower-level problem is inner approximated by copolyblocks in outcome space. From that, the upper-level problem is solved using the branch-and-bound method. Finding the bounds is converted to pseudoconvex programming problems, which are solved using the neurodynamic method. The algorithm's convergence is proved, and computational experiments are implemented to demonstrate the accuracy of the proposed approach.
\end{abstract}


\section{Introduction}

The bilevel optimization problem was first examined with mathematical models by Bracken and McGill \cite{Bracken1973}, following a formulation in the form of the Stackelberg game \cite{Stackelberg1934}. The focal difference between them and the prevalent  single-level optimization problem is the hierarchical structure of the problem, where constraints for an optimization problem are affected by another problem. The affected problem is called the lower-level problem, while the other is called the upper-level problem. Despite getting attention lately, it has been crucial because of its application to practical problems. 

Unlike single-level optimization problems, which can be imagined as one agent trying to achieve objectives alone, bilevel ones are more like two agents interacting with each other by their former decision while following objectives. Furthermore, since the majority of realistic problems are affected by many agents, the number of bilevel optimization research has increased quickly in the last decades, and they have spread in all fields of application. Some practical ones such as the Nash equilibrium problem in economics, optimal operation control problem \cite{Kalashnikov2018}, management problem in supply chain systems \cite{Gao2011,Lukac2008}, traffic and transportation network design problem \cite{Ukkusuri2013,Angulo2014}, machine learning problem \cite{Franceschi2018,MacKay2019},... 
There are two main ways to classify bilevel optimization problems. One considers their type of functions, such as linear, convex, nonconvex, or discrete ones.  Others consider the context of the problem, like to cooperate optimistic problem in contrast to the pessimistic problem, optimization over the efficient set, multiobjective bilevel optimization problem, multi leaders/followers problem, bilevel with fuzzy and stochastic extension. Although in different contexts, they still share the same bilevel optimization formation. 

In solving method aspect, exact algorithms with classic approaches were researched for decades, in contrast to the heuristic group with later approaches. The exact methods that should be mentioned are extreme-point approaches, branch-and-bound, complementary pivoting, descent methods, penalty function methods, or trust-region methods. And heuristic groups contain evolutionary approaches, local search-based or neural network-based ones. Lu et al. wrote a conscientious survey about these approaches (see \cite{Lu2016}).

This paper considers a bilevel optimization problem in the form of optimization over an efficient set scenario. Also, in the lower-level problem, we examine vector function, a multiobjective optimization problem that makes an overall formulation called semivectorial bilevel optimization. This topic has been researched for years, with the main target being linear functions or convex ones in lower level or upper level \cite{Belkhiri2021,Horst1999,Jorge2005,Muu2000, Shehu2019, Horst2007,Le2003,Ren2016}. Popular approaches were the penalty method, which converts the problem into a single-level optimization problem \cite{White1993,Le2003, Fischer2022}, and the branch-and-bound method, which solves iterated relaxed subproblems for boundings \cite{Belkhiri2021,Horst1999,Muu2000}. Recently there have been some nonconvex ones such as \cite{Mitsos2008, Dempe2021}, but the algorithms are all developed using heuristic methods. Our research considers pseudoconvex objective functions on both levels and proposes an algorithm. By utilizing the nice properties of functions, the algorithm is developed based on monotonic optimization combined with a recent neurodynamic approach, where the solution set of the lower-level problem is inner approximated by copolyblocks in outcome space. Afterwards, the branch-and-bound technique is used to resolve the upper-level problem. The neurodynamic approach is then used to solve the resulting pseudoconvex programming problem for finding the bounds. 

In the next section, we present the equivalent outcome space problem of the considered one. Section 3 gives some mathematics bases for building the proposed algorithm. Section 4 presents the proposed algorithm. We prove the algorithm's convergence in Section 5 before showing some special cases of the upper-level objective function. Computational experiments are shown in Section 6, and the conclusion is ended in the last section.
\section{Preliminaries}
\subsection{Notations and definitions}
The positive orthant cone in $p$-dimension space an its interior are denoted as 
$\mathbb{R}_{+}^{p}:=\{x\in\mathbb{R}^{p}\mid x\geq0\}$
and ${\rm int}\mathbb{R}_{+}^{p}$, respectively. Let's recall the definition of \emph{nondominated point} by considering two vectors
$a,b\in\mathbb{R}^{p}$, where $a\leq b$ if all components
$a_{i}\leq b_{i}$, $i=1,2,\dots,p$ and similarly $a<b$ if $a_{i}<b_{i}$
for all $i=1,2,\dots,p$. Let $Q\subset\mathbb{R}^{p}$ be some nonempty set.
A point $q$ in $Q$ is considered an \textit{nondominated point} if there is no other point $q^\prime$ in $Q$ such that $q^\prime$ is less than or equal to $q$. Similarly, a point $q$ in $Q$ is considered a \textit{weakly nondominated point} if there is no other point $q^\prime$ in $Q$ such that $q^\prime$ is less than $q$. The set of all nondominated points and weakly nondominated
points of $Q$ are denoted ${\rm Min}Q$ and ${\rm WMin}Q$, respectively.

\begin{definition} 
(Clarke, 1983 \cite{Clarke1983}). Suppose $\varphi$ is a function that maps from $\mathbb{R}^n$ to $\mathbb{R}$ and is locally Lipschitz near a point $x$ in $\mathbb{R}^n$. We can define the generalized directional derivative of $\varphi$ at $x$, in the direction of a vector $v$ in $\mathbb{R}^n$, as follows
$$
\varphi^{\circ}(x ; v)=\limsup _{y \rightarrow x, t \downarrow 0} \frac{\varphi(y+t v)-\varphi(y)}{t} \text {. }
$$
\end{definition} 
The Clarke's generalized subgradient of $\varphi$ at $x$ is given by $\partial \varphi(x)=\left\{\xi \in \mathbb{R}^n: \varphi^{\circ}(x ; v) \geq \xi^{\mathrm{T}} v\right.$, for all $v$ in $\left.\mathbb{R}^n\right\}$.

\begin{definition} (Clarke, 1983 \cite{Clarke1983}). Suppose that $\varphi: \mathbb{R}^n \rightarrow \mathbb{R}$ is locally Lipschitz near $x \in \mathbb{R}^n$. The one-side directional derivative for any direction $v \in \mathbb{R}^n$ is defined by
$$
\varphi^{\prime}(x ; v)=\lim _{t \downarrow 0} \frac{\varphi(x+t v)-\varphi(x)}{t} \text {, }
$$
and we say $\varphi$ is regular at $x$, if the one-side directional derivative exists and for all $v \in \mathbb{R}^n, \varphi^{\circ}(x ; v)=\varphi^{\prime}(x ; v)$. Moreover, $\varphi$ is said to be regular on $\mathbb{R}^n$ provided $\varphi$ is regular at any $x \in \mathbb{R}^n$.
\end{definition} 


\begin{definition} (Clarke, 1983 \cite{Clarke1983}). Let $S \subseteq \mathbb{R}^n$ be a convex set, a function $\varphi: S \rightarrow \mathbb{R}$ is said to be a real-valued convex function provided that, for all $x, x^{\prime} \in S$ and $\lambda \in[0,1]$, one has $\varphi\left(\lambda x+(1-\lambda) x^{\prime}\right) \leq \lambda \varphi(x)+(1-\lambda) \varphi\left(x^{\prime}\right)$.
If $\varphi: K \rightarrow \mathbb{R}$ is convex, then $\partial \varphi(x)=\left\{\xi \in \mathbb{R}^n: \varphi(x)-\varphi\left(x^{\prime}\right) \leq \xi^{\mathrm{T}}\left(x-x^{\prime}\right)\right.$, for all $\left.x^{\prime} \in \mathbb{R}^n\right\}$.
\end{definition} 
It’s important to know that convex functions are regular, as stated in Proposition 2.3.6 in Clarke's 1983 work \cite{Clarke1983}.
\begin{definition}
(Penot \& Quang, 1997 \cite{Pen1997}). Given a nonempty convex subset $S$ of $\mathbb{R}^n$, a function $\varphi: K \rightarrow \mathbb{R}^n$ is defined as pseudoconvex on $S$ if and only if for any two distinct points $x$ and $x^\prime$ in $S$ satisfies
$$
\exists \eta \in \partial \varphi(x): \eta^{\mathrm{T}}\left(x^{\prime}-x\right) \geq 0 \Rightarrow \varphi\left(x^{\prime}\right) \geq \varphi(x)
$$
\end{definition} 

\begin{definition} [Cambini \& Martein, 2008 \cite{Cambini2008}] Let $G: \mathbb{R}^n \rightarrow \mathbb{R}$ be locally Lipschitz, then
\begin{itemize}
    \item [(1)]  $G$ is pseudoconvex if for any $x, \tilde{x} \in \mathbb{R}^n$,
$\exists \xi \in \partial G(x): \xi^{\top}(\tilde{x}-x) \geq 0 \Rightarrow G(\tilde{x}) \geq G(x)$;
\item [(2)] $G$ is quasiconvex if for any $x, \tilde{x} \in \mathbb{R}^n$ and $\lambda \in[0,1]$, $G(\lambda x+(1-\lambda) \tilde{x}) \leq \max \{G(x), G(\tilde{x})\}$.
\end{itemize}
\end{definition} 

\begin{theorem} [Cambini \& Martein, 2008 \cite{Cambini2008}] \label{pseudoconvex1}
    Let $z(x) = \dfrac{f(x)}{g(x)}$ be the ratio of two differentiable functions $f$ and $g$ defined on an open convex set $S \subset \mathbb{R}^n$.
    \begin{itemize}
        \item [(1)]  If $f$ is a convex function and $g$ is both positive and affine, then $z$ is a pseudoconvex function;
        \item [(2)] If $f$ is non-negative and convex, and $g$ is positive and concave, then $z$ is pseudoconvex;
    \end{itemize}
\end{theorem} 
\begin{theorem}
    [Carosi \& Martein, 2006 \cite{Carosi2006}] \label{pseudoconvex2}
    Consider the ratio function $h(x) = \dfrac{1}{2}\dfrac{x^TAx}{(b^Tx + b_0)^3}$ where $A$ is a non-zero $n \times n$ symmetric matrix, $b \in \mathbb{R}^n$ and $b \in \mathbb{R}$. The function $h$ is pseudoconvex on $S = \{x \in \mathbb{R}^n: b^Tx + b_0 >0 \}$ if and only if $h$ is of the following form:
    \begin{equation*}
        h(x) = \dfrac{1}{2} \dfrac{\mu(b^Tx)^2}{(b^Tx + b_0)^3} \text{ where $b_0 < 0$}.
    \end{equation*}
    
\end{theorem}
\subsection{Problem formulation} \label{sec: prob_form}
The pseudoconvex semivectorial bilevel optimization problem under consideration can be expressed in the following manner
\begin{align}
    \min \  &{h}(x,y) \tag*{(BP)}\label{prob_SBP} \\
    \text{s.t.} \quad & {g}(x,y) \leq 0, y\in\mathbb{R}_{+}^{m}, \nonumber\\ 
                     & x \in {\text{Argmin}}\{f(x) \mid x\in X\},  \nonumber
\end{align}
where it will be assumed that

\noindent \textbf{(A1)} $X = \{x \in \mathbb{R}^n \mid s(x)\leq 0 \}$ is a nonempty bounded and convex set; 

\noindent \textbf{(A2)} The real-valued function ${h}: \mathbb{R}^m \times \mathbb{R}^n \rightarrow \mathbb{R}$ is continuous and the vector-valued functions ${g}: \mathbb{R}^m \times \mathbb{R}^n \rightarrow \mathbb{R}^{\ell}, {f}: \mathbb{R}^n \rightarrow \mathbb{R}^p, {s}: \mathbb{R}^n \rightarrow \mathbb{R}^q$ are continuously differentiable, where the integer numbers $m, n, \ell, p , q \geq 2$;

\noindent \textbf{(A3)} The objective functions $h, f$ are pseudoconvex and regular; 

\noindent \textbf{(A4)} The constraint functions $g, s$ are quasiconvex.

In the above formulation, the scalar function ${h}$ is called the upper-level objective function and the vectorial function ${f} = (f_1, f_2, \dots, f_p)$ is called the lower-level objective function. For convenience, as usual, we denote $X_{WE}$ as the weakly efficient solution set of the vector optimization problem $\min\{f(x) | x \in X \}$. Therefore, the constraint $x \in {\text{Argmin}}\{f(x) \mid x\in X\}$ can be replaced by $x \in X_{WE}$.

Problem \ref{prob_SBP}
covers a large class of complex nonconvex optimization problems, such as nonconvex multiplicative programming, optimization over the efficient set, and meta-learning problems in machine learning \cite{Franceschi2018,MacKay2019, tran2023framework, hoang2022improving}. The restricted form of  \ref{prob_SBP} for optimizing over the efficient set has been both studied in some previous work such as \cite{Tuy2016, Kim2016, thang2016solving, thang2016outcome, thang2015outcome} with the effective algorithms.

\subsection{Neurodynamic method for solving nonsmooth pseudoconvex programming problem} \label{sec:neuraldynamic}
We now consider a pseudoconvex programming problem as follows:
\begin{align} 
    \label{prob_SQ}
    \min\quad &r(x) \tag*{\ensuremath{(SQ)}}  \\ 
    s.t. \quad & x \in X,\nonumber 
\end{align}
where the objective function $r(x): \mathbb{R}^n \rightarrow \mathbb{R}$ is nonsmooth pseudoconvex, the constraint set $X = \{x \in \mathbb{R}^n\mid s(x)\leq 0\}$, $s(x):\mathbb{R}^n \rightarrow \mathbb{R}^m$ and $s_i(i=1, \ldots, m): \mathbb{R}^n \rightarrow \mathbb{R}$ are quasiconvex and differentiable. 

To tackle this problem, we employ the neurodynamic approach introduced by Liu \textit{et al}. \cite{Liu2021}. This method differs from others such as \cite{Malek2010, Bian2018, Li2014, Yu2019} because it permits using quasiconvex functions for inequality constraints. Even though the zero-level sets of both quasiconvex and convex functions are convex and a quasiconvex constraint can be substituted with an equivalent convex one, identifying such a replacement is typically not straightforward.\\
\indent To solve Problem \ref{prob_SQ},  Liu initially defines the set-valued function  $\Psi: \mathbb{R} \rightrightarrows [0,1]$, which is characterized by:
\begin{equation} \label{eq: Psi}
   \Psi(\xi)=\left\{
\begin{aligned} 
1, &\quad \xi>0 ; \\
[0,1], &\quad \xi=0 ; \\
0, &\quad \xi<0 .
\end{aligned}
\right.
\end{equation}

In fact, $\Psi(\xi)=\partial \max \{0, \xi\}$, and $\Psi$ is upper semicontinuous on $\mathbb{R}$.
Define the following function:
\begin{equation} \label{eq: S_m(x)}
   S_m(x)=\sum_{i=1}^m \max \left\{0, s_i(x)\right\} \text {, } 
\end{equation}
where $s_i(x),\ i=1,2, \ldots, m$ are given in \ref{prob_SQ}. The closed-form for $\partial S_m(x)$ is derived as
$$
\partial S_m(x)= \begin{cases}0, & x \in \operatorname{int}\left(X\right); \\ \sum_{i \in I_0(x)}[0,1] \partial s_i(x), & x \in \operatorname{bd}\left(X\right); \\ \sum_{i \in I_{+}(x)} \partial s_i(x)+\sum_{i \in I_0(x)}[0,1] \partial s_i(x), & x \notin X,\end{cases}
$$
where $I_{+}(x)=\left\{i \in\{1,2, \ldots, m\}: s_i(x)>0\right\}$, $I_0(x)=\left\{i \in\{1,2, \ldots, m\}: s_i(x)=0\right\}$, $\operatorname{int}\left(X\right)$ and $\operatorname{bd}\left(X\right)$ denote the interior and boundary of the feasible set $X$, respectively.

The neurodynamic model for solving Problem \ref{prob_SQ} are described in the form of a dynamic system as follows:
\begin{equation}
    \label{model}
    \begin{cases}
        x(0) &\in X; \\
        \dfrac{d}{d t} x(t) &\in-c(x(t)) \partial r(x(t))-\partial S(x(t)),
    \end{cases}
\end{equation}
where
\begin{equation} \label{eq: c(x)}
c(x(t))=\left\{\prod_{i=1}^{m} c_i(t) \mid c_i(t) \in 1-\Psi\left(s_i(x(t))\right), i=1,2, \ldots, m\right\}.
\end{equation}
In the neurodynamic model \eqref{model}, the state $x(0)$ is any initial state within $X$. The second expression is specified from the gradient descent method where the $-\partial S(x(t))$ term  serves to steer the state towards the feasible region $X$, while the term $-\partial r(x(t))$ causes the state to follow a descent direction of the objective function $r$. The $c(x(t))$ term adjusts the step-size of the state in model \eqref{model}. If the state is outside of $X$, the model will direct it towards $X$. If the state is within $X$, however, the model will direct it towards a feasible descent direction of $r$.

\begin{theorem} \cite[Theorem 4.3]{Liu2021}
The state $x(t)$ of
neural network \eqref{model} with any $x(0) \in X$, converges to an optimal solution of Problem \ref{prob_SQ}.
\end{theorem}

\section{The equivalent problem in outcome space}\label{sec:3}
We denote the outcome set $\mathcal{Z}:=\{z\in\mathbb{R}^{p}\mid z=f(x),x\in X\}$ and $\mathcal{G}=\left\{(x,y)\mid x\in X,y\in\mathbb{R}_{+}^{m},g(x,y)\le0\right\}$.
Since $X$ is bounded and $f$ consists of continuous functions, $\mathcal{Z}$
is bounded. We thus can find a box $[m,M]$ which contains $\mathcal{Z}$,
i.e., $m\leq\mathcal{Z}\leq M$. To determine $m$, it is sufficient
to solve the following problems for each component $m_{i},i=1,2,\dots,p$,
\begin{equation}\label{Pm_i}
\min\;f_{i}(x),\;\;{\rm s.t.}\;x\in X.    \tag*{\ensuremath{(P^m_i)}}
\end{equation}
Since $f_{i}(x)$ is a pseudoconvex function, any local minimum
of the above equation is also a global optimum \cite{Avriel1988}. It is straight forward to apply the neurodynamic model presented in Section \ref{sec:neuraldynamic} to \ref{Pm_i} as follow:
\begin{equation} \tag*{\ensuremath{Solve(P^m_i)}}
    \begin{cases}
        x(0) &\in X; \nonumber\\ 
        \dfrac{d}{d t} x(t) &\in-c(x(t)) \partial f_i(x(t))-\partial S_m(x(t)),
    \end{cases}
\end{equation}
where $S_m(x),\ c(x(t))$ are defined from \eqref{eq: S_m(x)} and \eqref{eq: c(x)}. However, the similar reasoning cannot be applied to find $M$ because the problem
$\max\{ f_{i}(x),\;\;{\rm s.t.}\;x\in X \}$ maximizing a pseudoconvex function over a convex set that is nonconvex. For
that purpose, we bound $X$ in a simplex $\Delta$ with vertex set
$V(\Delta)=\{\Delta^{0},\Delta^{1},\dots,\Delta^{n}\}$, where $\Delta^{0}=\min\{x\in\mathbb{R}^{n}\mid x\in X\}$
and $\Delta^{i}=(\Delta_{1}^{i},\Delta_{2}^{i},\dots,\Delta_{n}^{i}),i=1,2,\dots,n$
are defined by

\[
\Delta_{k}^{i}=\begin{cases}
\Delta_{k}^{0}, & {\rm if}\;k\neq i;\\
U-\sum_{j\neq k}\Delta_{j}^{0}, & {\rm if}\;k=i,
\end{cases}
\]
where $U$ is the optimal value of the problem: $\max\{\langle e,x\rangle\mid x\in X\},e\in\mathbb{R}^{n},e=(1,1,\dots,1)^{T}$.
This definition combined with the fact that $f$ is a pseudoconvex function leads to $X\subset\Delta$ and therefore a way to
specify $M$, such as
\begin{align*} \label{PM_i}  \tag*{\ensuremath{(P^M_i)}}
M_{i}: & =\max\{f_{i}(x)\mid x\in V(\Delta)\}\\
 & =\max\{f_{i}(x)\mid x\in\Delta\}\\
 & \geq\max\{f_{i}(x)\mid x\in X\},i=1,2,\dots,p.
\end{align*}

We can see that $\mathcal{Z}$ is not full-dimensional and nonconvex. We therefore define $\mathcal{Z}^{+}:=\mathcal{Z}+\mathbb{R}_{+}^{p}=\{z\in\mathbb{R}^{p}\mid\exists z^{0}\in\mathcal{Z},z^{0}\leq z\}$ which is a full-dimentional convex set. One problem of $\mathcal{Z}^{+}$ is that it is not bounded, we then bound the set $\mathcal{Z}$ by another equivalently efficient
set denoted as $\mathcal{Z}^\diamond$,
\begin{align*}
\mathcal{Z}^{\diamond}:=\; & \mathcal{Z}^{+}\cap(M-\mathbb{R}_{+}^{p}).
\end{align*}
The efficient equivalence of $\mathcal{Z}, \mathcal{Z}^+$ and $\mathcal{Z}^{\diamond}$ is shown in Proposition \ref{lem_MinY} where 
${\rm Min}\mathcal{Z},\ {\rm Min}\mathcal{Z}^{+}$ and ${\rm Min}\mathcal{Z}^{\diamond}$ denote the efficient sets of $\mathcal{Z},\ \mathcal{Z}^+$ and $\mathcal{Z}^\diamond$, respectively.
\begin{proposition}[\cite{Thang2020}] \label{lem_MinY} We have
i) $\;{\rm Min}\mathcal{Z}={\rm Min}\mathcal{Z}^{+}={\rm Min}\mathcal{Z}^{\diamond}$;

ii) ${\rm WMin}\mathcal{Z}={\rm WMin}\mathcal{Z}^{+}\cap\mathcal{Z}={\rm WMin}\mathcal{Z}^{\diamond}\cap\mathcal{Z}$.

\end{proposition}
We will use some concepts of monotonic optimization presented in \cite{Thang2020}. A real-valued function $d$ is said to be increasing (decreasing) over $S \in \mathbb{R}^n$ if and only if $d(x) \geq d(y)\ (d(x) \leq d(y),\ \forall x \in y + \mathbb{R}_+^n$. A set $Q\in\left[m,M\right]$ is called \textit{normal} if $\left[m,z\right]\in Q,\forall z\in Q$ and \textit{conormal} if $\left[z,M\right]\in Q,\forall z\in Q$. It is obvious that $\mathcal{Z}^\diamond$ is a conormal set. We then convert Problem \ref{prob_SBP} into a monotonic problem by constructing
a function $\varphi: \mathcal{Z}^{\diamond} \rightarrow\mathbb{R}$ as follows
\begin{equation*}\tag*{\ensuremath{MP(z)}} \label{eq: MP(z)}
  \varphi(z)=\min\{h(x,y)\mid (x,y)\in \mathcal{G},f(x)\leq z\}.  
\end{equation*}

As we can see, when $z$ increases, $\varphi(z)$ decreases due to
the expansion of decision space $X$, so $\varphi$ is a decreasing
function over the conormal set $\mathcal{Z}^\diamond$.

\begin{proposition}\label{prop-QWP_Y}The monotonic Problem \ref{prob_SBP} is equivalent
to the outcome-space problem
\begin{equation}
\min\{\varphi(z)\mid z\in{\rm WMin}\mathcal{Z}^\diamond\},\tag*{(OP)}\label{prob_OP}
\end{equation}
which means that \ref{prob_SBP} is solved if and only if \ref{prob_OP} is solved.
\end{proposition}\begin{proof} As presented in Section \ref{sec: prob_form}, the problem \ref{prob_SBP} is equivalent to
$$\min\{h(x,y)\mid g(x,y)\le0,x\in X_{WE},y\in\mathbb{R}_{+}^{m}\},$$
whose constraints can be rewritten as,
$$\{ x\in X,y\in\mathbb{R}_{+}^{m} \mid g(x,y)\le0,f(x)\in{\rm WMin}\mathcal{Z} \}$$
which can be reduced as follows,
\begin{align*}
    \{(x,y)\in\mathcal{G} \mid f(x) \in {\rm WMin}\mathcal{Z}\} &= \{(x,y)\in\mathcal{G} \mid f(x) \in {\rm WMin}\mathcal{Z}^\diamond \cap \mathcal{Z} \} \\
    &= \{(x,y)\in\mathcal{G} \mid f(x) \in {\rm WMin}\mathcal{Z}^\diamond \},
\end{align*}
On the one hand, suppose $\bar{z}$ is optimal solution of $\min\{\varphi(z)\mid z\in{\rm WMin}\mathcal{Z}^\diamond\}$ and $(\bar{x},\bar{y})$ is optimal solution of \textit{MP($\bar{z}$)}. Since $\varphi(\bar{z})$ reaches its optimal value at $(\bar{x},\bar{y})$, we get
\begin{equation}
\begin{cases}
\varphi(\bar{z})=h(\bar{x},\bar{y});\\
f(\bar{x})\le\bar{z};\\
h(\bar{x},\bar{y})\le h(x,y) & \forall (x,y)\in\mathcal{G},f(x)\le\bar{z}.
\end{cases}\label{eq:hx_1}
\end{equation}

thus we have
\[
h(\bar{x},\bar{y})\le h(x,y)\;\;\;\;\;\forall (x,y)\in\mathcal{G},x \in {\rm Wmin} \mathcal{Z}^\diamond.
\]

Therefore $(\bar{x},\bar{y})$ is also optimal solution of \ref{prob_SBP}.\\
On the other hand, suppose we have a solution $(\bar{x},\bar{y})$ of Problem \ref{prob_SBP}, we then have $f(\bar{x}) \in {\rm Wmin}\mathcal{Z}^\diamond$ and $\varphi(z)$ reaches its optimal value if and only if $z = f(\bar{x})$ implying $\min \varphi(z) = \varphi(f(\bar{x}))$ implying $z = f(\bar{x})$ is a solution to \ref{prob_OP}.





Therefore, we can solve \ref{prob_OP} in order to globally solve
\ref{prob_SBP}.\end{proof}

\section{The branch-and-bound scheme for solving Problem (OP)}
\subsection{Cutting cones and inner approximations of $\mathcal{Z}^\diamond$}
In our previous work \cite{Thang2020}, it is established that the union of an arbitrary union of normal (respectively, conormal) sets is itself a normal (respectively, conormal) set. The union of all normal (respectively, conormal) sets contained in $Q$ is denoted as the normal (respectively, conormal) hull of $Q$, represented by $\mathcal{N}\left(Q\right)$ (respectively, $\mathcal{M}\left(Q\right)$), which is also the minimal conormal set containing $Q$. A polyblock $\mathcal{P}$ is defined as the normal hull of a finite set of vertices $V$ within the interval $\left[m,M\right]$, such that $\mathcal{P}=\bigcup_{v\in V}\left[m,v\right]$ or equivalently $\mathcal{P}=\mathcal{N}\left(V\right)$. A copolyblock $\mathcal{P}$ is defined as the conormal hull of a finite set of vertices $V$ within the interval $\left[m,M\right]$, such that $\mathcal{P}=\bigcup_{v\in V}\left[v,M\right]$ or equivalently $\mathcal{P}=\mathcal{M}\left(V\right)$.

Let $v$ be a vertex in $\mathcal{N}(Q)$. $v$ is \textit{proper} if $\nexists v^\prime \in \mathcal{N}(Q)$ such that $v^\prime \neq v$ and $v^\prime \geq v$. An \textit{improper} vertex in $\mathcal{N}(Q)$ is not proper. A polyblock is determined by its set of proper vertices as it is the normal hull of its proper vertices. Similarly, a copolyblock is the conormal hull of its proper vertices where a proper vertex $v$ satisfies $\nexists v^\prime \in \mathcal{M}(Q)$ such that $v^\prime \neq v$ and $v^\prime \leq v$.

In our research, we developed an inner approxiamation algorithm, so instead
of determining a outer copolyblock by its proper vertices $V$, we
regard the proper vertices as ones of the opposite polyblock and construct the inner copolyblock as follows
\[
\mathcal{L}\left(V\right)\coloneqq\left[m,M\right]\setminus{\rm int}\left(\mathcal{N}\left(V\right)-\mathbb{R}_{+}^{p}\right).
\]

And for convenience, throughtout our paper, copolyblock is understood
as one built by this way, the inner copolyblock $\mathcal{L}\left(V\right)$
constructed above generates new proper vertices and initial proper
vertices $V$ is now called \textit{co-proper} vertices. Now we recall some main properties of copolyblocks in following proposition.

\begin{proposition}[\cite{Thang2020}]\label{prop:copolyblock}
\begin{enumerate}
    \item[(i)] A finite union of copolyblocks is a copolyblock.
    \item[(ii)] The union of a finite number of conormal set forms a compact conormal set.
    \item[(iii)]  The intersection of a finite number of copolyblocks forms a compact conormal set.
\end{enumerate}
\end{proposition}

\begin{proposition}\label{prop:new_v}
Given $v\in\mathcal{Z}^{\diamond}$ and $w_{v}$ determined by Remark \ref{remark:P_0_v}, the new copolyblock $\mathcal{P}'$ obtained by applying the cutting cone of $\mathcal{\mathcal{Z}^{\diamond}}$ on $\mathcal{P}$ at $w_{v}$ has vertex set $V',$ where $\mathcal{P}$ is a copolyblock in the box $[m,M]$ with co-proper vertex set $V$ such that $\mathcal{P}\subseteq\mathcal{\mathcal{Z}^{\diamond}} $, and 
\[
V'=\ensuremath{(V\setminus\{v\})\cup\{v-(v_{i}-w_{i})e^{i}\},\quad i=1,\dots,p.}
\]
\end{proposition}

According to Proposition \ref{prop:copolyblock}  (iii) any compact conormal set can be approximated as closely as desired by a copolyblock. As a result, a family of copolyblocks can be used to approximate the compact conormal set $\mathcal{Z}^{\diamond}$. In particular, Proposition \ref{prop:new_v} generates a nested sequence of copolyblocks that inner-approximates the outcome set $\mathcal{Z}^{\diamond}$, such that $\mathcal{P}^{0}\subset\mathcal{P}^{1}\subset\mathcal{P}^{2}\subset\dots\subset\mathcal{P}^{k}\subset\mathcal{P}^{k+1}\subset\dots\subset\mathcal{Z}^{\diamond}$, where the initial copolyblock $\mathcal{P}^{0}=[M,M]$ is constructed by calculating $M$ as in Section \ref{sec:3}.

In each iteration, the process where copolyblock $\mathcal{P}^{k+1}$
is generated from $\mathcal{P}^{k}$ and the resulting inner approximation
is described in Procedure CopolyblockCut,

\begin{algorithm}[H] 
\SetAlgorithmName{Procedure}{procedure}{List of Procedures}
\caption{\textit{CopolyblockCut}} \label{algo:next_time_step}
\KwIn{A copolyblock $V^k$ or a set of its vertices, the current considering vertex $v^k$ and the intersection $w^k$ of the ray starting from $v^k$ along direction $\hat{d}$ and $\partial\mathcal{Z^+}$}
\KwOut{The new inner approximate outcome set $V^{k+1}$}
Set $V^{k+1}\gets V^{k}\setminus\{v^{k}\}$.\\
\For{$i \gets 1$ \textbf{\textup{to}} $p$}
{$z^{i} = v^{k}-(v^{k}_{i}-w^{k}_{i})e^{i}$;\\
\If{$z^{i}_{i}\neq m_{i}$}
{$V^{k+1}\gets(V^{k+1}\cup\{z^{i}\})$.}}
\end{algorithm}

\subsection{Brand-and-bound algorithm scheme}
Utilizing the outcome space approach, the solution of \ref{prob_SBP} is attained through the enhancement of upper and lower bounds for the objective function subsequent to each iteration. Furthermore, the outcome space is approximated recurrently via the cutting cone methodology on inner copolyblocks. Commencing with copolyblock $\mathcal{P}^{0}=[M,M]$, we construct a sequence of copolyblocks ${\mathcal{P}^{k}}$ iteratively such that
$$\mathcal{P}^{0}\subset\mathcal{P}^{1}\subset\mathcal{P}^{2}\subset\dots\subset\mathcal{P}^{k}\subset\mathcal{P}^{k+1}\subset\dots\subset\mathcal{Z}^{\diamond}.$$
The subsequent notations shall be employed:
\begin{itemize}
\item The set of all co-proper vertices is denoted by $V^{k}$ which defines the copolyblock $\mathcal{P}^{k}=\mathcal{L}\left(V^{k}\right)$ .
\item The upper and lower bounds are denoted by $\alpha_{k}$ and $\beta_{k}$, respectively.
\end{itemize}
At the initial step with $k=0$, we have $V^{0}=\left\{ M\right\} ,\mathcal{P}^{0}=\left[M,M\right],\alpha_{0}=+\infty$. In a typical iteration $k$, by solving $\min\left\{ \varphi\left(v\right)\mid v\in V^{k}\right\} $
the lower bound $\beta_{k}$ is assigned, then we determine $v^{k}$
such that $\beta_{k}=\varphi\left(v^{k}\right)$. By solving $\left(P^{2}\left(v^{k}\right)\right)$,
we find a new weakly efficient point $z^{k}=f\left(x^{k}\right)$
of the lower problem, after that we find solution $y^k$ for problem \ref{prob_SBP} with the $x^k$ found, if the problem is feasible, we will obtain a solution $\left(x^{k},y^{k}\right)$.
After that we compare $h\left(x^{k},y^{k}\right)$ to
$\alpha_{k}$ to update upper bound. If $\alpha_{k}$ and $\beta_{k}$
satisfy terminate condition meaning that upper bound is close enough
to lower bound, the algorithm stop and return $\left(\left(x^{k},y^{k}\right),h\left(x^{k},y^{k}\right)\right)$.
Otherwise, a new co-proper vertices set $V^{k+1}$ is created by procedure
CopolyblockCut, before new inner approximate outcome set $\mathcal{P}^{k+1}=\mathcal{L}\left(V^{k+1}\right)$
is determined and next iteration upper bound is assigned.

\subsection{Determining the upper bounds}

A weakly nondominated point of $\mathcal{Z}^+$ can be easily determined
by the following remark.

\begin{remark}\label{remark:P_0_v}Let fix a vector $\hat{d}>0$ in
$\mathbb{R}^{n}$ and $v$ an arbitrary point in $\mathbb{R}^{p}$.
Then the intersection $w_{v}$ of the line through $v$ along direction
$\hat{d}$ can be determined by 
\begin{equation} \label{eq:bar_w}
    w_{v}=v+t_{v}\hat{d}
\end{equation} where $t_{v}$
is the optimal value of the following problem
\begin{equation}
\begin{array}{rl}
\min & t\tag*{\ensuremath{(P^{0}(v))}}\\
{\rm s.t.} & v+t\hat{d}\in\mathcal{Z}^{+},\;t\in\mathbb{R}.
\end{array}\label{eq:P_0_v}
\end{equation}
\end{remark}
Lemma \ref{lem_bar_w} shows that $w_{v}:=v+t_{v}\hat{d}$ is a weakly nondominated point
of $\mathcal{Z}^{+}$.
\begin{lemma}\cite[Lemma 2.1]{Thang2020} \label{lem_bar_w} 
    For any point $v$ in $\mathbb{R}^{p}$.
    Then there exists the unique point $w_{v}$ determined by \eqref{eq:bar_w}
    that is a weakly nondominated point of $\mathcal{Z}^{+}$.
    
\noindent 
\end{lemma}.

We can rewrite \ref{eq:P_0_v} in its explicit form
\begin{align}
\min\,\, & t\tag*{\ensuremath{(P^{1}(v))}}\label{eq:P_1_v}\\
\mbox{s.t.\,\,} & f(x)-t\hat{d}-v\leq0,\nonumber \\
 & x\in X,\;t\in\mathbb{R}.\nonumber 
\end{align}
It can be seen that \ref{eq:P_1_v} is, in general, nonconvex. However,
it is equivalent to the following problem

\begin{align}
\min\,\, & \max \left\{\dfrac{f_{j}(x)-v_{j}}{\hat{d}_{j}}\ \middle\vert\ j=1,...,p \right\}\tag*{\ensuremath{(P^{2}(v))}}\label{eq:Pbarv}\\
\mbox{s.t.}\;\; & x\in X.\nonumber 
\end{align}
which is solvable as the objective function is pseudoconvex
and the feasible set is convex.
\begin{lemma}\cite[Lemma 2.2 ]{Thang2020}\label{lem_equiv} 
    It can be inferred that problems \ref{eq:P_0_v} and \ref{eq:Pbarv} are equivalent. In other words, if Problem \ref{eq:P_0_v} has an optimal solution $(x^*,t^*)$ , then $x^*$ is the optimal solution of Problem has  \ref{eq:Pbarv}. Conversely, \ref{eq:Pbarv} has an optimal solution $x^*$ with the corresponding optimal value $t^*$, then $(x^*,t^*)$ is the optimal solution of \ref{eq:P_0_v}. Additionally, it should be noted that problem \ref{eq:Pbarv} is a pseudoconvex programming problem.
\end{lemma}

By Lemma \ref{lem_equiv}, finding a Pareto solution to Problem \ref{prob_SBP} is transformed into finding the optimal solution to a pseudoconvex programming problem. The gradient descent technique and the neurodynamic method are two ways to solve quasiconvex programming problems. Since the objective function of \ref{eq:Pbarv} is nonsmooth even though the function ${f}$ is smooth, the neurodynamic approach is used to tackle this problem.\\
For that, we calculate the subgradient of $\max \left\{\dfrac{f_{j}(x)-v_{j}}{\hat{d}_{j}}\ \middle\vert\ j=1,...,p \right\}$ with the help of Lemma \ref{lem:2.5} and Lemma \ref{lem:2.7}.

\begin{lemma} [Clarke, 1983 \cite{Clarke1983}, Chain rule] \label{lem:2.5}Let $V(x): \mathbb{R}^n \rightarrow \mathbb{R}$ be a regular function and $x(t): \mathbb{R} \rightarrow \mathbb{R}^n$ be Lipschitz near $t$ and differentiable at $t$. Then, for almost every $t \in[0,+\infty)$, it holds that $\dot{V}(x(t))=\zeta^{\mathrm{T}} \dot{x}(t)$ for all $\zeta \in \partial V(x(t))$.
\end{lemma} 

\begin{lemma} [Clarke, 1983 \cite{Clarke1983}] \label{lem:2.7} Let $\left\{f_i, i=1,2,3, \ldots, n \right\}$ be a finite set of functions that are Lipschitz near $x$ and regular. Define $\psi(x)=\max \left\{f_i(x), i=1,2,3, \ldots, n\right\}$. Then the subdifferential of $\psi(x)$ is given by $\partial \psi(x)=\operatorname{conv}\left\{\partial f_i(x), i \in I(x)\right\}$ where $I(x)$ is the set of indices $i$ such that $f_i(x)=\psi(x)$ and "conv" denotes the convex hull of a set.
\end{lemma} 
The dynamic system of the problem can then be formulated analogously to the problem in Section \ref{sec:neuraldynamic} by applying the theorems directly to \ref{eq:Pbarv} as follows,
\begin{align*}
    \partial \max \left\{\dfrac{f_{j}(x)-v_{j}}{\hat{d}_{j}}\ \middle\vert\ j=1,...,p \right\} 
    &= \text{conv} \left\{\partial \dfrac{f_{j}(x)-v_{j}}{\hat{d}_{j}} \ \middle\vert\ j\in I(x) \right\} \\
    & = \dfrac{1}{\hat{d_j}}\text{conv} \left\{\partial f_j(x(t))^\mathrm{T}\dot{x}(t) \right\}.
\end{align*}
We then attain the dynamic model for solving \ref{eq:Pbarv},
\begin{equation} \tag{\ensuremath{ND_{}P^2(v)}}
    \begin{cases}
        x(0) &\in X; \nonumber\\ 
        \dfrac{d}{d t} x(t) &\in-c(x(t)) \times \dfrac{1}{\hat{d_j}}\text{conv} \left\{\partial f_j(x(t))^\mathrm{T}\dot{x}(t) \right\}-\partial S_m(x(t)),
    \end{cases}
\end{equation}
where $c(x(t)$, $\Psi$ and $S_m(x(t))$ is defined in \eqref{eq: c(x)}, \eqref{eq: Psi} and \eqref{eq: S_m(x)}, respectively.
\subsection{Determining the lower bounds}
Since $h(x,y)$ is a pseudoconvex function, the subproblem of finding

\begin{equation} \tag{\ensuremath{LB(V^k)}} \label{prob_LB}
\beta_{k}=\min\left\{ \varphi(v)\mid v\in V^{k}\right\} ,\\
\end{equation}

in the proposed algorithm requires solving \eqref{eq: MP(z)} where $f_{i}$ are quasiconvex functions (by pseudoconvexity) and $\mathcal{G}$ is a nonempty
compact convex set with the quasiconvexity of $g$ in it. These problems can be handled by the neurodynamic model mentioned in Section \ref{sec:neuraldynamic}. We first define function $\tilde{G}(u)$ as,
\begin{equation}
    \tilde{G}(u) = \left(s_1(x),\dots,s_m(x),-y, f_1(x)-z_1, \dots, f_p(x) - z_p, g_1(x,y),\dots,g_l(x,y) \right)^T,
\end{equation}
where $u = (x, y)^T \in \mathbb{R}^{n+1}$ and $u \in \tilde{\mathcal{G}} := \mathcal{G} \cup \{x \in X \mid f(x) \leq z \}$.
Analogously to Section \ref{sec:neuraldynamic}, we then introduce a function $G_m$,
\begin{equation}
    \tilde{G_m}(u) = \sum_{i = 1}^{m + p+ l + 1} \max \{0, \tilde{G_i}(u) \}.
\end{equation}
We can calculate subgradient of $\tilde{G_m}(u)$ as
$$
\partial \tilde{G_m}(u)= \begin{cases}0, & u \in \operatorname{int}\left(\tilde{\mathcal{G}}\right); \\ \sum_{i \in I_0(u)}[0,1] \partial \tilde{G_i}(u), & u \in \operatorname{bd}\left(\tilde{\mathcal{G}}\right); \\ \sum_{i \in I_{+}(u)} \partial s_i(u)+\sum_{i \in I_0(u)}[0,1] \partial \tilde{G_i}(u), & u \notin \tilde{\mathcal{G}},\end{cases}
$$
where $I_{+}(u)=\left\{i \in\{1,2, \ldots, m\}: \tilde{G_i}(u)>0\right\}$, $I_0(u)=\left\{i \in\{1,2, \ldots, m\}: \tilde{G_i}(u)=0\right\}$, $\operatorname{int}\left(\tilde{\mathcal{G}}\right)$ and $\operatorname{bd}\left(\tilde{\mathcal{G}}\right)$ denote the interior and boundary of the feasible set $\tilde{\mathcal{G}}$, respectively.

The neurodynamic model for solving Problem \ref{eq: MP(z)} are constructed as,
\begin{equation} \tag{\ensuremath{NDMP(z)}}
    \label{SolveMP(z)}
    \begin{cases}
        u(0) &\in \tilde{\mathcal{G}}; \\
        \dfrac{d}{d t} u(t) &\in-c(u(t)) \partial h(u)-\partial \tilde{G_m}(u(t)),
    \end{cases}
\end{equation}
where
\begin{equation} 
c(u(t))=\left\{\prod_{i=1}^{m} c_i(t) \mid c_i(t) \in 1-\Psi\left(\tilde{G_i}(u(t))\right), i=1,2, \ldots, m\right\},
\end{equation}
where $\Psi$ is defined in \eqref{eq: Psi}.
\subsection{The description of the proposed algorithm}

We aim to find the approximate solutions to \ref{prob_SBP} and \ref{prob_OP}.
Let a small tolerance $\varepsilon>0$, a point $z^{*}\in{\rm WMin}\mathcal{Z}$
is called an $\varepsilon-$optimal solution to Problem \ref{prob_OP}
if there exists an upper bound $\alpha^{*}$ for Problem \ref{prob_OP}
such that $\alpha^{*}-\varphi(z^{*})<\varepsilon(1+\vert\varphi(z^{*})\vert)$.
Any $x^{*}\in X_{WE}$ satisfying $f(x^{*})\le z^{*}$ is then called
an approximate optimal solution to Problem \ref{prob_SBP}.


\begin{algorithm}[H]
\SetAlgorithmName{Algorithm}{procedure}{List of Procedures}
\renewcommand{\thealgocf}{}
\let\oldnl\nl
\newcommand{\nonl}{\renewcommand{\nl}{\let\nl\oldnl}}
\caption{\textit{Solve} \ref{prob_SBP}}
\KwIn{A pseudoconvex problem in form of \ref{prob_SBP}}
\KwOut{The approximate optimal solution with arbitrary error threshold}
Choose a sufficient small tolerance level $\varepsilon > 0$.
Solve problems \ref{Pm_i} and \ref{PM_i}, $i=1,\dots,p$ to determine
the box $[m,M]$.

Set $\mathcal{P}^{0}\gets[M,M]$, $V^{0}\gets\{M\}$ and choose a direction $\hat{d}\in\mathbb{R}^{p}_{+}$
(e.g., $\hat{d}=e$).

\For{$i \gets 1$ \textbf{\textup{to}} $p$}{
    $z^{i} = M-(M_{i}-m_{i})e^{i}$, solve \ref{SolveMP(z)} with $z = z^i$;\\
    \eIf{(MP($z^i$)) has a solution $(x^{i},y^{i})$}
    {$\alpha^{i}=h(x^{i},y^{i})$.}
    {$\alpha^{i}=\infty$.}
}

Set initial upper bound $\alpha_{0}\gets \min\{\alpha^{i}\mid i=1,2,\dots,p\}$, $k\gets0$ and boolean $update\gets False$.\\
Set current best solution $x^{*}$ corresponding to problem $\varphi(z^{i})$ which satisfies $\varphi(z^{i})=\alpha_{0}$.\\
\For {$k \gets 1$ \textup{\textbf{to}} $\infty$ }{
\SetAlgoLined
    \ForEach {$v\in V^{k}$}{
    Solve \eqref{SolveMP(z)} with $z = v$;\\
    \eIf{(\textit{MP(v)}) has a solution $(x^{*}, y^{*})$}
    {\eIf{$\nexists i:f_{i}(x^{*})=m_{i}$}
    {$\beta_{v}\gets \varphi(v)$.}{$V^{k} \gets V^{k}\setminus \{v\}$.}}
    {$V^{k} \gets V^{k}\setminus \{v\}$.}
    }

    Solve \eqref{prob_LB} to get a new lower bound $\beta_k$ 
    and $v^{k}\in V^{k}$ such that $\varphi(v^{k})=\beta_{k}$;
    
    Solve problem $(P^{2}(v^{k}))$ to find an optimal solution $(x^{k},t_{k})$ and set $w^{k} \gets v^{k}+t_{k}\hat{d}; z^{k}\gets f(x^{k})$;\\
    
    \If{$w^{k}=z^{k}$ \textbf{\textup{and}} $w^{k}_{i}>m_{i},\forall i$}
    {Find a feasible $y^{k}$ satisfying $(x^{k},y^{k})\in \mathcal{G}$;\\
    \eIf{$y^{k}$ is found}{$update \gets True$.}{$update \gets False$.}
    }

    \If{\textit{update} \textbf{\textup{and}} $h(x^{k},y^{k})<\alpha_{k}$}{Update the upper bound $\alpha_{k} \gets h(x^{k},y^{k}),\ x^{*}  \gets x^{k},\  y^{*}  \gets y^{k}$.}

    \uIf{$\alpha_{k}-\beta_{k}\leq\varepsilon (1+|\beta_k|)$}{\textbf{Terminate.}}
    \Else{Determine the new set $V^{k+1}$ by using Procedure \ref{algo:next_time_step}.
    }
    
    Determine the new inner approximate outcome set $\mathcal{P}^{k+1} \gets \mathcal{L}(V^{k+1})$;
    
    $\alpha_{k+1} \gets \alpha_{k}$.

}
\end{algorithm}

\bigskip

\section{The convergence of the proposed algorithm}
The convergence of the proposed algorithm, when $k$ is sufficiently large, is proven through the following lemmas.

\begin{lemma}\label{lem-max_norm}The number k tends to infinity
and
\[
\lim_{k\rightarrow\infty}\max_{v\in V^{k}}\left\Vert w_{v}-v\right\Vert =0,
\]
where $V^{k}$ denotes the set of all proper vertices that determine $\mathcal{P}^{k}$ and $w_{v}$ denotes the corresponding weakly nondominated point of $\mathcal{Z}^{+}$ obtained by solving Problem \ref{eq:Pbarv}.

\end{lemma}

\begin{proof} Consider a vertex $v^{k}\in\mathcal{P}^{k}$ chosen at the $k^{th}$
iteration and the optimal value $t_{k}$ of $\left(P^{2}\left(v^{k}\right)\right)$.
As in \eqref{eq:bar_w}, let $w_{v^{k}}=v^{k}+t_{k}\hat{d}$, we have
\begin{equation}
{\rm {\rm Vol}}\left[v^{k},w_{v}^{k}\right]=\left(t_{k}\right)^{p}{\rm Vol}\left[0,\hat{d}\right].\label{eq:Vol_d}
\end{equation}

The lemma stays valid if $\max_{v\in V^{k}}\left\Vert w_{v}-v\right\Vert =0$
at some $k\ge0$. Otherwise, there exists $v^{k}\in V^{k}$ such that
$\left\Vert w_{v^{k}}-v^{k}\right\Vert =\max_{v\in V^{k}}\left\Vert w_{v}-v\right\Vert >0$.
We also have $\mathcal{P}^{k}\subseteq\mathcal{P}^{k+1}\setminus\left(v^{k}-{\rm int}\mathbb{R}_{+}^{p}\right)$,
since $\left[v^{k},w_{v^{k}}\right]\subseteq\mathcal{P}^{k}$ deduced from the definition of $w_{v^{k}}$, the volume of $\mathcal{P}^{k}$
satisfies
\begin{equation}
{\rm Vol}\mathcal{P}^{k+1}-{\rm Vol}\mathcal{P}^{k}\ge{\rm Vol}\left[v^{k},w_{v}^{k}\right].\label{eq:Vol_P^k}
\end{equation}

Combining \eqref{eq:Vol_d} with \eqref{eq:Vol_P^k}, we obtain
\[
{\rm Vol}\mathcal{P}^{k+1}-{\rm Vol}\mathcal{P}^{k}\ge\left(t_{k}\right)^{p}{\rm Vol}\left[0,\hat{d}\right].
\]

Therefore,
\[
\sum_{i=0}^{k}\left({\rm Vol}\mathcal{P}^{i+1}-{\rm Vol}\mathcal{P}^{i}\right)\ge\left(\sum_{i=0}^{k}\left(t_{i}\right)^{p}\right){\rm Vol}\left[0,\hat{d}\right].
\]

We deduce
\[
{\rm Vol}\mathcal{Z}^{\diamond}\ge{\rm Vol}\mathcal{P}^{k+1}\ge{\rm Vol}\mathcal{P}^{k+1}-{\rm Vol}\mathcal{P}^{0}\ge\left(\sum_{i=0}^{k}\left(t_{i}\right)^{p}\right){\rm Vol}\left[0,\hat{d}\right],
\]
for all $k\ge1$. As $k$ approaches infinity, the positive series $\sum_{i=0}^{k}\left(t_{i}\right)^{p}$ is upper bounded by ${\rm Vol}\mathcal{Z}^{\diamond}/{\rm Vol}\left[0,\hat{d}\right]$, which implies its convergence and $\lim_{i\rightarrow\infty}t_{i}=0$. Since $\hat{d}$ is bounded, for any $i\ge1$, we have
\[
\lim_{i\rightarrow\infty}\max_{v\in V^{i}}\left\Vert w_{v}-v\right\Vert =\lim_{i\rightarrow\infty}\left\Vert w_{v^{i}}-v^{i}\right\Vert =\lim_{i\rightarrow\infty}t_{i}\left\Vert \hat{d}\right\Vert =0.
\]

\end{proof}

\begin{lemma}\label{lem-sol_at_border}When the solution $(\bar{x},\bar{y})$
of Problem \ref{prob_SBP} satisfies $\exists i:f_{i}\left(\bar{x}\right)=m_{i}$,
then $\bar{x}$ can be obtained by solving Problem $\varphi\left(\bar{z}\right)$
with $\bar{z}$ is the solution of the following problem
\[
\min\left\{ \varphi\left(z\right)\mid z\in\left\{ M-\left(M_{j}-m_{j}\right)e^{j}\mid j=1,2,\dots,p\right\} \right\} .
\]

\end{lemma}

\begin{proof}Since $(\bar{x},\bar{y})$ is the solution to the problem \ref{prob_SBP}, we have
\begin{equation}
h\left(\bar{x},\bar{y}\right)\leq\varphi\left(z\right),\quad\forall z\in\left\{ M-\left(M_{j}-m_{j}\right)e^{j}\mid j=1,2,\dots,p\right\} .\label{eq:h_phi}
\end{equation}

Consider $z^{i}=M-\left(M_{i}-m_{i}\right)e^{i}$, because $f_{i}\left(\bar{x}\right)=m_{i}$
and $f\left(\bar{x}\right)\leq M$, we have $f\left(\bar{x}\right)\leq z^{i}$
and therefore $\bar{x},\bar{y}$ is a feasible solution of the problem
\[
\varphi\left(z^{i}\right)=\min\{h(x,y)\mid (x,y)\in \mathcal{G},f(x)\leq z^{i}\},
\]
thus
\begin{equation}
\varphi\left(z^{i}\right)\leq h\left(\bar{x},\bar{y}\right).\label{eq:phi_h}
\end{equation}

From \eqref{eq:h_phi} and \eqref{eq:phi_h}, we must have
\[
h\left(\bar{x},\bar{y}\right)=\varphi\left(z^{i}\right).
\]

Because of that, $(\bar{x},\bar{y})$ can be obtained by selecting best solution
after solving $p$ problems $MP(z^j),\ j=1,2,\dots,p$.\end{proof}

\begin{lemma}\label{lem-sol_on_MinZ}For any $z\in\mathcal{Z}^{\diamond}$
such that \ref{eq: MP(z)} has a solution, then if we continuously
solve $\left(P^{2}\left(v^{k}\right)\right)$ with initial $v^{0}\equiv z$
to obtain $p$ new $v^{k}-(v_{i}^{k}-w_{i}^{k})e^{i}$ points adding
to $V^{k+1}$ and remove any $v\in V^{k+1}$ from $V^{k+1}$ if \textit{MP($v$)}
has no solution, two following statements will be true,
\begin{enumerate}
    \item[(i)] If $\nexists(\bar{x},\bar{y})\in \mathcal{G}$ such that $f\left(\bar{x}\right)\in{\rm Min}\mathcal{Z}\cap\left(z-\mathbb{R}_{+}^{p}\right)$, we will obtain $V^{k}=\emptyset$ when $k$ tends to infinity.
    \item[(ii)] If $\exists(\bar{x},\bar{y})\in \mathcal{G}$ such that $f\left(\bar{x}\right)\in{\rm Min}\mathcal{Z}\cap\left(z-\mathbb{R}_{+}^{p}\right)$,
then we can extract a sequence $\left\{ u^{k}\right\} _{k=0}^{\infty},u^{k}\in V^{k}$
with $u^{0}\equiv z$ such that if $w_{u^{k}}$ is the weakly nondominated
point of $\mathcal{Z}^{+}$ induced by the solution $\left(x^{k},t_{k}\right)$
of problem $(P^{2}(u^{k}))$, when $k$ tends to infinity
$$
\lim_{k\rightarrow\infty}\left\Vert w_{u^{k}}-f\left(\bar{x}\right)\right\Vert =0.
$$
\end{enumerate}
\end{lemma}

\begin{proof}i) Since $\nexists(\bar{x},\bar{y})\in \mathcal{G}$ such that $f\left(\bar{x}\right)\in{\rm Min}\mathcal{Z}\cap\left(z-\mathbb{R}_{+}^{p}\right)$,
by denoting
\begin{eqnarray*}
A & = & {\rm Min}\mathcal{Z}\cap\left(z-\mathbb{R}_{+}^{p}\right),\\
B & = & \left(z-\mathbb{R}_{+}^{p}\right)\cap\left\{ f\left(x\right)\mid (x,y)\in \mathcal{G}\right\} ,
\end{eqnarray*}
we have
\[
A\cap B=\emptyset.
\]

Therefore, we can assume that the infimum distance between these two
set which is Hausdorff distance is a number $\varepsilon>0$.

Also, because of Lemma \ref{lem-max_norm}, we can chose a number
$k$ such that
\[
\lim_{k\rightarrow\infty}\max_{v\in V^{k}}\left\Vert w_{v}-v\right\Vert <\varepsilon.
\]

Thus, since $w_{v}\in A,\forall v\in V^{k}$, it does not exist any
$v\in V^{k}$ such that $\exists b\in B:b\leq v$ due to the Hausdorff
distance $\varepsilon>0$.

As a consequence, $\nexists(x,y)\in \mathcal{G}$ such that
$f\left(x\right)\leq v,\forall v\in V^{k}$, hence \textit{MP($v$)}
attains no solution for any $v\in V^{k}$ and so $\forall v\in V^{k}$
is removed from $V^{k}$.

ii) We assume some $v^{k}\in V^{k}$ satisfying $v^{k}>f\left(\bar{x}\right)$
and consider $p$ points
\[
v^{k,i}=v^{k}-\left(v_{i}^{k}-w_{v^{k},i}\right)e^{i},i=1,2,\dots,p.
\]
We will prove that 
$\exists i:v_{i}^{k,i}\geq f_{i}\left(\bar{x}\right),$
if not, we assume
$v_{i}^{k,i}<f_{i}\left(\bar{x}\right),\forall i.$ Besides, we also have
$w_{v^{k},i}\leq v_{i}^{k,i},\forall i,$
therefore
$w_{v^{k},i}<f_{i}\left(\bar{x}\right),\forall i.$

We can conclude that $w_{v^{k}}<f\left(\bar{x}\right)$. However, this results in a contradiction since both $w_{v^{k}}$ and $f\left(\bar{x}\right)$ belong to ${\rm WMin}\mathcal{Z}^{\diamond}$. Therefore, the assumption is false and we have
\[
\exists i:v_{i}^{k,i}\geq f_{i}\left(\bar{x}\right).
\]

Also, since $v^{k}>f\left(\bar{x}\right)$ and $v_{j}^{k}=v_{j}^{k,i},\forall j\neq i$,
we have
\[
\exists i:v^{k,i}\geq f\left(\bar{x}\right).
\]

Given that $v^{k,i}$ is an element of $V^{k+1}$, we can use mathematical induction to derive a sequence $\left\{ u^{k}\right\} _{k=0}^{\infty}$ where $u^{k}$ is an element of $V^{k}$ and $u^{0}$ is equal to $z$. This sequence satisfies the condition that $u^{k}>f\left(\bar{x}\right)$ because $u^{0}=z>f\left(\bar{x}\right)$.

Now we will prove that, when $k$ tends to infinity
\[
\lim_{k\rightarrow\infty}\left\Vert w_{u^{k}}-f\left(\bar{x}\right)\right\Vert =0.
\]

Since $u^{k+1}=u^{k}-\left(u_{i}^{k}-w_{u^{k},i}\right)e^{i}$ with
some $i$, we have $u^{k+1}<u^{k}$, thus sequence $\left\{ u^{k}\right\} _{k=0}^{\infty}$
is a decreasing one and therefore it converges because of having lower
bound $f\left(\bar{x}\right)$.

If $\left\{ u^{k}\right\} _{k=0}^{\infty}$ converges at a point $u\in{\rm Min}\mathcal{Z}\cap\left(z-\mathbb{R}_{+}^{p}\right)$
such that $u>f\left(\bar{x}\right)$, we have a contradiction since
$u,f\left(\bar{x}\right)\in{\rm Min}\mathcal{Z}\cap\left(z-\mathbb{R}_{+}^{p}\right)$,
hence
\[
\lim_{k\rightarrow\infty}\left\Vert u^{k}-f\left(\bar{x}\right)\right\Vert =0.
\]

Combining with Lemma \ref{lem-max_norm}, for any $\varepsilon>0$,
we can find a $u^{t}\in\left\{ u^{k}\right\} _{k=0}^{\infty}$ such
that $\left\Vert w_{u^{t}}-u^{t}\right\Vert <\dfrac{\varepsilon}{2}$
and $\left\Vert u^{t}-f\left(\bar{x}\right)\right\Vert <\dfrac{\varepsilon}{2}$. 

We have triangle inequality
\[
\left\Vert w_{u^{t}}-f\left(\bar{x}\right)\right\Vert \leq\left\Vert w_{u^{t}}-u^{t}\right\Vert +\left\Vert u^{t}-f\left(\bar{x}\right)\right\Vert <\varepsilon.
\]

And therefore we must have
\[
\lim_{k\rightarrow\infty}\left\Vert w_{u^{k}}-f\left(\bar{x}\right)\right\Vert =0.
\]

\end{proof}

\begin{lemma}\label{lem-converge}At the $k^{th}$ iteration, let
$w_{v^{k}}$ be the weakly nondominated point of $\mathcal{Z}^{+}$
induced by the solution $\left(x^{k},t_{k}\right)$ of Problem $(P^{2}(v^{k}))$,
we consider the situation when $\nexists i:w_{v^{k},i}=m_{i}$, when
k tends to infinity
\[
\lim_{k\rightarrow\infty}\left\Vert \varphi(f(x^{k}))-\varphi(v^{k})\right\Vert =0.
\]

\end{lemma}

\begin{proof} Given that $f$, $g$, and $h$ are continuous functions with finite values and $X$ is a nonempty compact convex set, it follows that $\varphi$ is also a continuous function with finite values.

Besides, as a result of Lemma \ref{lem-max_norm}
\[
\lim_{k\rightarrow\infty}\left\Vert w_{v^{k}}-v^{k}\right\Vert \le\lim_{k\rightarrow\infty}\max_{v\in V^{k}}\left\Vert w_{v}-v\right\Vert =0.
\]

Since $\nexists i:w_{v^{k},i}=m_{i}$, thus $w_{v^{k}}\in{\rm Min}\mathcal{Z}$
and we have $w_{v^{k}}=v^{k}+t_{k}\hat{d}=f\left(x^{k}\right)$.

Because $\lim_{k\rightarrow\infty}\left\Vert w_{v^{k}}-v^{k}\right\Vert =0$
and $\varphi$ is a continuous function, we have
\[
\lim_{k\rightarrow\infty}\left\Vert \varphi(w_{v^{k}})-\varphi(v^{k})\right\Vert =0.
\]

Therefore
\[
\lim_{k\rightarrow\infty}\left\Vert \varphi(f(x^{k}))-\varphi(v^{k})\right\Vert =\lim_{k\rightarrow\infty}\left\Vert \varphi(w_{v^{k}})-\varphi(v^{k})\right\Vert =0.
\]
\end{proof}

\begin{theorem} If problem \ref{prob_OP} has a solution, then for any given $\varepsilon>0$, the algorithm will terminate after a finite number of iterations and return an $\varepsilon-$optimal solution to Problem \ref{prob_OP}.
\end{theorem}

\begin{proof}According to Lemma \ref{lem-sol_at_border}, if the
global solution $(\bar{x},\bar{y})$ of Proposition \ref{prop-QWP_Y} satisfies $\exists i:f_{i}\left(\bar{x}\right)=m_{i}$
which means$f\left(\bar{x}\right)\in{\rm WMin}\mathcal{Z}\setminus{\rm Min}\mathcal{Z}$,
we will have $\bar{x}$ obtained by the proposed algorithm.

On the other hand, when $f\left(\bar{x}\right)$ is on ${\rm Min}\mathcal{Z}$,
since $V^{k}$ is the set of vertices of inner approximate outcome set, we must
have some $v^{k}\in V^{k}$ such that $f\left(\bar{x}\right)\in{\rm Min}\mathcal{Z}\cap\left(v^{k}-\mathbb{R}_{+}^{p}\right)$.
Applying Lemma \ref{lem-sol_on_MinZ}.ii), it exists a sequence pair
$\left(u^{t},w_{u^{t}}\right)$ such that
\[
\lim_{k\rightarrow\infty}\left\Vert w_{u^{k}}-f\left(\bar{x}\right)\right\Vert =0,
\]
thus
\[
\lim_{k\rightarrow\infty}\left\Vert \varphi\left(w_{u^{k}}\right)-\varphi\left(f\left(\bar{x}\right)\right)\right\Vert =0,
\]
so we can find $k>0$ such that
\[
\left\Vert \varphi\left(w_{u^{k}}\right)-\varphi\left(f\left(\bar{x}\right)\right)\right\Vert <\frac{\varepsilon}{2}.
\]

Also, thanks to Lemma \ref{lem-converge}, we have
\[
\lim_{k\rightarrow\infty}\left\Vert \varphi(w_{u^{k}})-\varphi(u^{k})\right\Vert =0.
\]

According to the construction of the bound $\alpha_{k}$ and $\beta_{k}$,
we have
\[
0\le\alpha_{k}-\beta_{k}=h(x^{k},y^{k})-\varphi(v^{k})=\varphi(f(x^{k}))-\varphi(v^{k}).
\]

In case we chose $v^{k}=u^{k}$, such that
\[
\left\Vert \varphi\left(f\left(\bar{x}\right)\right)-\varphi(u^{k})\right\Vert <\frac{\varepsilon}{2},
\]
we will obtain
\begin{eqnarray*}
0 & \leq & \alpha_{k}-\beta_{k}\\
 & = & \varphi(f(x^{k}))-\varphi(f(u^{k}))\\
 & = & \varphi(w_{u^{k}})-\varphi(f(u^{k}))\\
 & \leq & \Big\| \varphi\left(w_{u^{k}}\right)-\varphi\left(f\left(\bar{x}\right)\right)\Big\| +\Big\| \varphi\left(f\left(\bar{x}\right)\right)-\varphi(u^{k})\Big\| \\
 & < & \varepsilon.
\end{eqnarray*}

Furthermore, Lemma \ref{lem-sol_on_MinZ}.i) proves that if we select
a vertex $v^{k}\in V^{k}$ which causes \textit{MP($v^k$)} to have no feasible solution in ${\rm Min}\mathcal{Z}\cap\left(v^{k}-\mathbb{R}_{+}^{p}\right)$,
the vertex set $V^{k}$ will become empty in a finite number of iterations
and we can continue for other $v^{k}\in V^{k}$. Moreover, if we choose
a vertex $v^{k}\in V^{k}$ such that 
\[
\nexists (x,y)\in\mathcal{G}:f\left(x\right)\in\left(v^{k}-\mathbb{R}_{+}^{p}\right),
\]
problem \textit{MP($v^k$)} will have no solution resulting in the elimination of $v^{k}$ from $V^{k}$ and the algorithm goes on for other
another vertex.

As a result, we can conclude that the algorithm terminates in a finite number of iterations and $x^{*}$ is an $\varepsilon-$optimal solution to Problem
\ref{prob_OP}.

\end{proof}

\section{Computational experiments}

We demonstrate the efficiency of our algorithm through numerical results
in this section. The experiments were performed on 2.6GHz Intel Core
i7 (four logical cores), 8Gb of RAM. The code is implemented on Matlab
R2018a.

\begin{example}\label{ex_benson2011}Consider problem \ref{prob_SBP}
where functions are defined by
\[
\begin{array}{ccl}
h(x) & = & x_{1}+x_{2}^{2}\\
f_{1}(x) & = & x_{1}^{2}+x_{2}^{2}+0.4x_{1}-4x_{2}\\
f_{2}(x) & = & \max\{-0.5x_{1}-0.25x_{2}-0.2,-2x_{1}+4.6x_{2}-5.8\}
\end{array}
\]

and
\[
X=\{x\in\mathbb{R}^{2}\mid Ax\le b,x\ge0,c(x)\le0\},
\]

where
\[
A=\left[\begin{array}{cc}
1.0 & -2.0\\
-1.0 & 1.0\\
2.0 & 1.0\\
2.0 & 5.0\\
-1.0 & -1.0
\end{array}\right],b=\left[\begin{array}{c}
1.0\\
1.0\\
4.0\\
10.0\\
-1.5
\end{array}\right],
\]

and
\[
c(x)=0.5(x_{1}-1)^{2}+1.4(x_{2}-0.5)^{2}-1.1.
\]

\end{example}
\begin{table}

\caption{\label{tab:ex_benson2011} The computational result of Example \ref{ex_benson2011}}
\centering
\begin{tabular}{ccccc}
\hline 
$k$ & $v^{k}$ & $\alpha^{k}$ & $\beta^{k}$ & Gap\tabularnewline
\hline 
\hline 
1 & (0.432442, 1.067557) & 1.572121 & 1.250000 & 0.322121\tabularnewline
\hline 
2 & (0.573591, 0.926408) & 1.431824 & 1.250000 & 0.181824\tabularnewline
\hline 
3 & (0.688942, 0.811057) & 1.346757 & 1.250000 & 0.096757\tabularnewline
\hline
& &\dots & &  \tabularnewline
\hline
31 & (0.996542, 0.503457) & 1.250012 & 1.250000 & 0.000012\tabularnewline
\hline 
32 & (0.997561, 0.502438) & 1.250006 & 1.250000 & 0.000006 \tabularnewline
\hline 
\end{tabular}
\end{table}
Observe that $h(x),\ f_1(x)$ are convex functions and hence pseudoconvex, $f_2(x)$ is the maximum of two affine functions so that it is also a pseudoconvex function. The restraint functions are linear so that they all satisfy assumption \textbf{(A4)}. At the initialization step, the box $[m,M]=[-2.52,0.49,3.29,1.20]$,
the initial upper bound $\alpha_{0}=\infty$, and the initial lower
bound $\beta=\varphi(M)=1.25$ where $\varphi$ is the constructed
function. It is natural to choose $\hat{d}=(1,1)^{T}$ and $\varepsilon=0.00001$.
The algorithm terminates in 1.2682 seconds after 32 iterations and
returns the optimal solution $x^{*}=(0.997561, 0.502439)$ with objective
value $h(x^{*})=1.250006$. This result is better than those reported
in \cite{Benson2012}, which are $x=(0.2585,1.2415)$ and $h(x)=1.7989$. The computation details are given in Table \ref{tab:ex_benson2011}, and more close to the exact optimal results which are $x^* = (1.0, 0.5)$ with
$h(x^*) = 1.25$.

\begin{example}Consider the following fractional programming
problem where $h(x)$ and $f_2(x)$ is pseudoconvex due to Theorem \ref{pseudoconvex1} and \ref{pseudoconvex2}, respectively. \label{ex:benson2005}

\noindent 
\begin{align*}
    {\rm min}\  h(x) = \dfrac{2x_1 + 3x_2}{4x_1 + 5x_2 + 10}\\
\end{align*}

where
\begin{align*}
    f_1(x) = \dfrac{(3x_1 +x_2)^2}{(3x_1 + x_2 - 1)^3}\\
    f_2(x) = \dfrac{x_{1}^{2}-2x_{1}+x_{2}^{2}-8x_{2}}{x_{2}+1}
\end{align*}
and the lower-level problem is $\min (f_1(x), f_2(x))$ with the feasible set $$X = \{x_{1},x_{2}\in\mathbb{R} \vert 2x_{1}+x_{2} \leq 6, 3x_{1}+x_{2}  \leq  8, x_{1}-x_{2}  \leq 1, x_{1},x_{2}  \geq 1\}$$
\end{example}
At the initialization step, the box $[m, M]$ = $[0.19,-4.34,0.42,-3.67]$, the initial upper bound $\alpha_{0}= \infty$, and the initial lower bound $\beta=\varphi(m)=0.276170$ where $\varphi$ is the constructed function. We choose $\hat{d}=(1,1)^{T}$ and $\varepsilon=0.01$. The algorithm terminates in 0.8376 seconds after 6 iterations and returns the optimal solution $x^* = (1.020247,1.835256)$ with objective
value $h(x^*) = 0.302334$. The computation details are given in Table \ref{tab:ex_benson2005}.

\begin{table}[h]
\centering{}\caption{\label{tab:ex_benson2005} The computational result of Example \ref{ex:benson2005}}
\begin{centering}
\begin{tabular}{ccccc}
\hline 
$k$ & $v^{k}$ & $\alpha^{k}$ & $\beta^{k}$ & Gap\tabularnewline
\hline 
\hline 
1 &  (2.095784, 1.712646) &  0.346225 & 0.276170 & 0.070055\tabularnewline
\hline 
2 & (1.781338, 2.206333) & 0.346225 & 0.285986 & 0.070055\tabularnewline
\hline 
3 & (1.454576, 2.020253) & 0.346067 & 0.285986 & 0.060081\tabularnewline
\hline 
4 & (1.161954, 1.887563) & 0.331592 & 0.31198 & 0.030346\tabularnewline
\hline 
5 & (1.300950, 1.946424) & 0.331592 & 0.315449 & 0.019604\tabularnewline
\hline 
6 & (1.020247, 1.835256) & 0.324469 & 0.315449 & 0.009020\tabularnewline
\hline 
\end{tabular}
\par\end{centering}
\end{table}

\begin{example} We consider the following problem where the lower problem has four objective functions
\noindent  \label{ex:ex3}
\begin{align*}
    \min &\ h(x) = \dfrac{3x_1 + 2x_2 + 10x_3 +11}{x_1 + x_2 +x_3 +10}\\
\end{align*}
where 

\[
\begin{array}{ccl}
    f_1(x) & = & \dfrac{2x_{1}+5x_{2}+3x_{3}+10}{3x_{2}+3x_{3}+10} \\
    f_2(x) & = & \dfrac{2x_{1}+4x_{2}+10}{4x_{1}+4x_{2}+5x_{3}+10} \\
    f_3(x) & = & \dfrac{x_{1}+2x_{2}+5x_{3}+10}{x_{1}+5x_{2}+5x_{3}+10}\\
    f_4(x) & = & \dfrac{x_{1}+2x_{2}+4x_{3}+10}{5x_{2}+4x_{3}+10}
\end{array}
\]
The feasible set for the lower-level problem is the set of positive vectors in $\mathbb{R}^3_+$ such that
\begin{align*}
    2x_{1}+x_{2}+5x_{3} \leq 10, \\
    x_{1}+6x_{2}+3x_{3} \leq 10,\\ 
    5x_{1}+9x_{2}+2x_{3} \leq 10,\\ 
    9x_{1}+7x_{2}+3x_{3} \leq 10
\end{align*}
\end{example}
All objective functions are pseudoconvex by Theorem \ref{pseudoconvex1}. At the initialization step, the box  $[m,M] =[1.00,0.50,1.65,0.79,1.17,1.00,3.00,1.00]$, $\hat{d}=(1, 1, 1, 1)^{T}$ and $\varepsilon=0.01$.
The initial upper bound $\alpha_{0}=\infty$, and the initial lower
bound $\beta=\varphi(M)=0.604$.
The algorithm terminates in 1.5260 seconds after 7 iterations and
returns the optimal solution $x^{*}=(0.000000, 0.775596, 0.007336)$ with objective
value $h(x^{*})=0.607448$. The computation details are given in Table \ref{tab:ex3}.

\begin{table}[h]
\centering{}\caption{\label{tab:ex3} The computational result of Example \ref{ex:ex3}}
\begin{centering}
\begin{tabular}{ccccc}
\hline 
$k$ & $v^{k}$ & $\alpha^{k}$ & $\beta^{k}$ & Gap\tabularnewline
\hline 
\hline 
1 & (0.000000, 0.483159, 0.697862) & 0.894430 & 0.604000 & 0.290430 \tabularnewline
\hline 
2 & (0.000000, 0.783986, 0.483034) & 0.818089 & 0.603704 & 0.214385 \tabularnewline
\hline 
3 & (0.000000, 0.877863, 0.084834) & 0.648966 & 0.603704 & 0.045262\tabularnewline
\hline 
4 & (0.000000, 0.822004, 0.292454) & 0.648966 & 0.603704 & 0.045262\tabularnewline
\hline 
5 & (0.000000, 0.825505, 0.127032) &0.648966 & 0.603704 & 0.045262\tabularnewline
\hline
6 & (0.000000, 0.823124, 0.223619) & 0.648966 & 0.603704 & 0.045262\tabularnewline
\hline 
7 & (0.000000, 0.775596, 0.007336) &  0.607448 & 0.603704 & 0.003745\tabularnewline
\hline 
\end{tabular}
\par\end{centering}
\end{table}

\begin{example} We consider an example in \cite{example4}\label{ex:ex4}with linear objective in the upper problem
    \[
\begin{array}{ccl}
h(x) & = & -x_{1} - 0.9\\
f_{1}(x) & = & x_1\\
f_{2}(x) & = & x_2\\
g(x) &=& x_1^2 + x_2^2 - 0.81
\end{array}
\]
and the feasible set for the lower-level problem is 
\[
X = \{x \in [-1, 1] \vert x_1 + x_2 + 1 \geq 0 \}
\]
\end{example}
At the initialization step, the box  $[m,M] =[-0.90,-0.90,0.00,-0.00]$, $\hat{d}=(1, 1)^{T}$ and $\varepsilon=0.01$.
The initial upper bound $\alpha_{0}=\infty$, and the initial lower
bound $\beta=\varphi(M)=-1.8$ where $\varphi$ is the constructed
function.
The algorithm terminates in 1.4638 seconds after 17 iterations and
returns the optimal solution $x^{*}=(-0.893699, -0.106301)$ with objective
value $h(x^{*})=-1.793699$ which is significantly close to the exact optimal value reported in \cite{example4}. The computation details are given in Table \ref{tab:ex4}.
\begin{table}[h]
\centering{}\caption{\label{tab:ex4} The computational result of Example \ref{ex:ex4}}
\begin{centering}
\begin{tabular}{ccccc}
\hline 
$k$ & $v^{k}$ & $\alpha^{k}$ & $\beta^{k}$ & Gap\tabularnewline
\hline 
\hline 
1 & (-0.526290, -0.473709) & -1.426291 & --1.800000 & 0.373709\tabularnewline
\hline 
2 & (-0.775599, -0.224400) & -1.675600 & -1.800000 & 0.124400\tabularnewline
\hline 
3 & (-0.893699, -0.106300) & -1.793699 & -1.800000 & 0.006301\tabularnewline
\hline 
\end{tabular}
\par\end{centering}
\end{table}

\begin{example} We consider another example in \cite{example4}
\label{ex:ex5}
\[
\begin{array}{ccl}
    h_1(x) & = & \left(x_1-1\right)^2+\displaystyle\sum_{i=2}^{14} x_i^2+0.25\\
     f_{1}(x) & = &\displaystyle\sum_{i=1}^{14} x_i^2\\
    f_{2}(x) & = & \left(x_1-0.5\right)^2+\sum_{i=2}^{14} x_i^2\\
\end{array}
\]

\[
    X =\{ x \in \mathbb{R}^{14} \vert -1 \leq\left(x_1, x_2, \ldots, x_{14}\right) \leq 2 \}
\]
\end{example}
At the initialization step, the box  $[m,M] =[0.00,0.00,0.25,0.25]$, $\hat{d}=(1, 1)^{T}$ and $\varepsilon=0.01$.
The initial upper bound $\alpha_{0}=\infty$, and the initial lower
bound $\beta=\varphi(M)=0.5$ where $\varphi$ is the constructed
function.
The algorithm terminates in 1.4638 seconds after 17 iterations and
returns the optimal solution $x^{*}=(0.5, 0,0,0,0,0,0,0,0,0,0,0,0,0)$ with objective
value $h(x^{*})=0.5$ which is also the exact optimal value reported in \cite{example4}.  The computation details are given in Table \ref{tab:ex5}.
\begin{table}[h]
\centering{}\caption{\label{tab:ex5} The computational result of Example \ref{ex:ex5}}
\begin{centering}
\begin{tabular}{cccc}
\hline 
$k$  & $\alpha^{k}$ & $\beta^{k}$ & Gap\tabularnewline
\hline 
\hline 
1  & 0.812500 & 0.500000 & 0.312500\tabularnewline
\hline 
2  & 0.566406 & 0.500000 & 0.066406\tabularnewline
\hline 
3  & 0.503922 &0.500000 & 0.003922\tabularnewline
\hline 
4  & 0.500015 & 0.500000 & 0.000015\tabularnewline
\hline 
5  & 0.500000 &0.500000 & 0.000000\tabularnewline
\hline 
\end{tabular}
\par\end{centering}
\end{table}

\begin{example} We now consider the following pseudoconvex optimization problem  
\noindent  \label{ex:ex6}
\begin{align*}
    \min \ h(x,y) = \dfrac{x_1^2 + 2x_2^2 +10y_1^2 + y_2^2 +11}{x_1 + x_3 + y_1 + 20} 
\end{align*} 

\[
\begin{array}{ccl}
    f_1(x) & = & \dfrac{2x_{1}+5x_{2}+3x_{3}+10}{3x_{2}+3x_{3}+10} \\
    f_2(x) & = & \dfrac{2x_{1}+4x_{2}+10}{4x_{1}+4x_{2}+5x_{3}+10} \\
    f_3(x) & = & \dfrac{x_{1}+2x_{2}+5x_{3}+10}{x_{1}+5x_{2}+5x_{3}+10}\\
    f_4(x) & = & \dfrac{x_{1}+2x_{2}+4x_{3}+10}{5x_{2}+4x_{3}+10}
\end{array}
\]
\begin{align*}
    X = &\{ x \in \mathbb{R}^3_+ \vert 2x_{1}+x_{2}+5x_{3} \leq 10,\\
     &x_{1}+6x_{2}+3x_{3} \leq 10, 5x_{1}+9x_{2}+2x_{3} \leq 10, 9x_{1}+7x_{2}+3x_{3} \leq 10\} 
\end{align*}
\[
\begin{array}{ccl}
    g_1(x) &= -x_2 -x_3-2y_1 -y_2 + 2   \\
     g_2(x)&= x_2 + x_3 -5y_1 +2y_2 - 1 \\
     y \in \mathbb{R}^2_+
\end{array}
\]
\end{example}
At the initialization step, the box  $[m,M] =[1.14,0.68,0.99,0.98,1.27,0.92,1.14,1.17]$, $\hat{d}=(1, 1, 1, 1)^{T}$ and $\varepsilon=0.01$.
The initial bounds $\alpha_{0}=\infty$, $\beta=\varphi(M)=0.009170$.
The algorithm terminates after 5 iterations and
returns the optimal solutions $x^{*}= (0.130662, 0.156198, 1.558087)$ and $y^*$ = $(0.142857, 0.000000)$ with objective
value $h(x^{*}, y^*)= 0.012365$. The computation details are given in Table \ref{tab:ex6}.
\begin{table}[h]
\centering{}\caption{\label{tab:ex6} The computational result of Example \ref{ex:ex6}}
\begin{centering}
\begin{tabular}{ccccc}
\hline 
$k$ & $v^{k}$ & $\alpha^{k}$ & $\beta^{k}$ & Gap\tabularnewline
\hline 
\hline 
1 & ( 0.000000, 0.422286, 1.291999,  0.142857, 0) & 0.026160 & 0.009170 & 0.190285 \tabularnewline
\hline 
2 & ( 0.001831, 0.421485, 1.292799, 0.142857, 0) & 0.026094 & 0.009170 & 0.016923 \tabularnewline
\hline 
3 & (0.329682, 0.276641, 1.437644, 0.142857, 0) & 0.021261 & 0.009170 & 0.012091 \tabularnewline
\hline 
4 & (0.464042, 0.170190, 1.544095, 0.142857, 0) & 0.021261 & 0.009170 & 0.012091 \tabularnewline
\hline 
5 & (0.130662, 0.156198, 1.558087, 0.142857, 0) & 0.012365 & 0.009170 & 0.003195 \tabularnewline
\hline
\end{tabular}
\par\end{centering}
\end{table}

\section{Conclusion}
This study presents an exact algorithm for solving a lass of pseudoconvex semivectorial bilevel optimization problem where the upper objective function is scalar pseudoconvex and ower one is a pseudoconvex vector function. The algorithm’s efficiency stems from its use of the outcome space cutting cone through a monotonic optimization approach and the neural dynamic approach to solve subproblems, which is computationally efficient and does not necessitate particular characteristics of the objective functions. The algorithm’s convergence has been demonstrated and its effectiveness is supported by computational experiments.
\section*{Acknowledgement}

This work was supported by Vietnam Ministry of Education and Training.

\section*{Disclosure statement}
No potential conflict of interest was reported by the authors







\bibliographystyle{tfs}
\bibliography{ref}

\begin{thebibliography}{10}
\providecommand{\MR}{\relax\unskip\space MR }
\providecommand{\url}[1]{\normalfont{#1}}
\providecommand{\urlprefix}{Available at }

\bibitem{Cambini2008}
L.M. Alberto~Cambini, \emph{Generalized Convexity and Optimization}, 1st ed.,
  Springer Berlin, Heidelberg, 2008.

\bibitem{Angulo2014}
E. Angulo, E. Castillo, R. Garcia-Rodenas, and J. Sanchez-Vizcaino, \emph{A
  continuous bi-level model for the expansion of highway networks}, Computers
  \& Operations Research 41 (2014), pp. 262--276.

\bibitem{Avriel1988}
M. Avriel, W.E. Diewert, S. Schaible, and I. Zang, \emph{Generalized
  Concavity}, Society for Industrial and Applied Mathematics, 2010.

\bibitem{Belkhiri2021}
H. Belkhiri, M.E.A. Chergui, and F.Z. Oua{\"i}l, \emph{Optimizing a linear
  function over an efficient set}, Operational Research 22 (2022), pp.
  3183--3201.

\bibitem{Benson2012}
H.P. Benson, \emph{An outcome space algorithm for optimization over the weakly
  efficient set of a multiple objective nonlinear programming problem}, Journal
  of Global Optimization 52 (2012), pp. 553--574.

\bibitem{Bian2018}
W. Bian, L. Ma, S. Qin, and X. Xue, \emph{Neural network for nonsmooth
  pseudoconvex optimization with general convex constraints}, Neural Networks
  101 (2018), pp. 1--14.

\bibitem{Bracken1973}
J. Bracken and J.T. McGill, \emph{Mathematical programs with optimization
  problems in the constraints}, Operations Research 21 (1973), pp. 37--44.

\bibitem{Carosi2006}
L. Carosi and L. Martein, \emph{Some Classes of Pseudoconvex Fractional
  Functions via the Charnes-Cooper Transformation}, in \emph{Generalized
  Convexity and Related Topics}, Berlin, Heidelberg. Springer Berlin
  Heidelberg, 2006, pp. 177--188.

\bibitem{Clarke1983}
F.H. Clarke, \emph{Optimization and Nonsmooth Analysis}, Society for Industrial
  and Applied Mathematics, 1990.

\bibitem{Fischer2022}
A. Fischer, A.B. Zemkoho, and S. Zhou, \emph{Semismooth newton-type method for
  bilevel optimization: global convergence and extensive numerical
  experiments}, Optimization Methods and Software 37 (2022), pp. 1770--1804.

\bibitem{Franceschi2018}
L. Franceschi, P. Frasconi, S. Salzo, R. Grazzi, and M. Pontil, \emph{Bilevel
  Programming for Hyperparameter Optimization and Meta-Learning}, in
  \emph{Proceedings of the 35th International Conference on Machine Learning},
  J. Dy and A. Krause, eds., Proceedings of Machine Learning Research Vol.~80,
  10--15 Jul. PMLR, 2018, pp. 1568--1577.

\bibitem{Gao2011}
Y. Gao, G. Zhang, J. Lu, and H.M. Wee, \emph{Particle swarm optimization for
  bi-level pricing problems in supply chains}, Journal of Global Optimization
  51 (2011), pp. 245--254.

\bibitem{Stackelberg1934}
J.R. Hicks, \emph{{Marktform und Gleichgewicht}}, The Economic Journal 45
  (1935), pp. 334--336.

\bibitem{hoang2022improving}
L.P. Hoang, D.D. Le, T.A. Tran, and T.T. Ngoc, \emph{Improving pareto front
  learning via multi-sample hypernetworks}, arXiv preprint arXiv:2212.01130
  (2022).

\bibitem{Horst2007}
R. Horst, N.V. Thoai, Y. Yamamoto, and D. Zenke, \emph{On optimization over the
  efficient set in linear multicriteria programming}, Journal of Optimization
  Theory and Applications 134 (2007), pp. 433--443.

\bibitem{Horst1999}
R. Horst and N.V. Thoai, \emph{Maximizing a concave function over the efficient
  or weakly-efficient set1research supported by the "deutsche
  forschungsgemeinschaft" through the "graduiertenkolleg mathematische
  optimierung" at the university of trier and through the project decomp.1},
  European Journal of Operational Research 117 (1999), pp. 239--252.

\bibitem{Jorge2005}
J.M. Jorge, \emph{A bilinear algorithm for optimizing a linear function over
  the efficient set of a multiple objective linear programming problem},
  Journal of Global Optimization 31 (2005), pp. 1--16.

\bibitem{Kalashnikov2018}
V. Kalashnikov, S. Dempe, B. Mordukhovich, and S.V. Kavun, \emph{Bilevel
  optimal control, equilibrium, and combinatorial problems with applications to
  engineering}, Mathematical Problems in Engineering 2017 (2017), p. 7190763.

\bibitem{Kim2016}
N. Kim and T. Tran~Ngoc, \emph{Optimization over the efficient set of a
  bicriteria convex programming problem}, Pacific Journal of Optimization 9
  (2013).

\bibitem{Le2003}
H.A. Le~Thi, T. Pham~Dinh, and L.D. Muu, \emph{Simplicially-constrained dc
  optimization over efficient and weakly efficient sets}, Journal of
  Optimization Theory and Applications 117 (2003), pp. 503--531.

\bibitem{Li2014}
Q. Li, Y. Liu, and L. Zhu, \emph{Neural network for nonsmooth pseudoconvex
  optimization with general constraints}, Neurocomputing 131 (2014), pp.
  336--347.

\bibitem{Liu2021}
N. Liu, J. Wang, and S. Qin, \emph{A one-layer recurrent neural network for
  nonsmooth pseudoconvex optimization with quasiconvex inequality and affine
  equality constraints}, Neural Networks 147 (2022), pp. 1--9.

\bibitem{Lu2016}
J. Lu, J. Han, Y. Hu, and G. Zhang, \emph{Multilevel decision-making: A
  survey}, Information Sciences 346-347 (2016), pp. 463--487.

\bibitem{Lukac2008}
Z. Lukac, K. Soric, and V.V. Rosenzweig, \emph{Production planning problem with
  sequence dependent setups as a bilevel programming problem}, European Journal
  of Operational Research 187 (2008), pp. 1504--1512.

\bibitem{MacKay2019}
M. Mackay, P. Vicol, J. Lorraine, D. Duvenaud, and R. Grosse, \emph{Self-Tuning
  Networks: Bilevel Optimization of Hyperparameters using Structured
  Best-Response Functions}, in \emph{International Conference on Learning
  Representations}. 2019.

\bibitem{Malek2010}
A. Malek, N. Hosseinipour-Mahani, and S. Ezazipour, \emph{Efficient recurrent
  neural network model for the solution of general nonlinear optimization
  problems}, Optimization Methods and Software 25 (2010), pp. 489--506.

\bibitem{Mitsos2008}
A. Mitsos, P. Lemonidis, and P.I. Barton, \emph{Global solution of bilevel
  programs with a nonconvex inner program}, Journal of Global Optimization 42
  (2008), pp. 475--513.

\bibitem{Muu2000}
L. Muu, \emph{On the construction of initial polyhedral convex set for
  optimization problems over the efficient set and bilevel linear programs},
  Vietnam Journal of Mathematics 28 (2000), pp. 77--82.

\bibitem{Pen1997}
J.P. Penot and P.H. Quang, \emph{{Generalized Convexity of Functions and
  Generalized Monotonicity of Set-Valued Maps}}, Journal of Optimization Theory
  and Applications 92 (1997), pp. 343--356.

\bibitem{Ren2016}
A. Ren and Y. Wang, \emph{A novel penalty function method for semivectorial
  bilevel programming problem}, Applied Mathematical Modelling 40 (2016), pp.
  135--149.

\bibitem{Shehu2019}
Y. Shehu, P.T. Vuong, and A. Zemkoho, \emph{An inertial extrapolation method
  for convex simple bilevel optimization}, Optimization Methods and Software 36
  (2019), pp. 1--19.

\bibitem{example4}
A. Sinha, P. Malo, and K. Deb, \emph{Towards Understanding Bilevel
  Multi-objective Optimization with Deterministic Lower Level Decisions}, in
  \emph{Evolutionary Multi-Criterion Optimization}, Cham. Springer
  International Publishing, 2015, pp. 426--443.

\bibitem{Dempe2021}
A.Z. Stephan~Dempe, \emph{Bilevel Optimization: Advances and Next Challenges},
  1st ed., Springer Cham, 2020.

\bibitem{thang2015outcome}
T.N. Thang, \emph{Outcome-based branch and bound algorithm for optimization
  over the efficient set and its application}, in \emph{Some Current Advanced
  Researches on Information and Computer Science in Vietnam: Post-proceedings
  of The First NAFOSTED Conference on Information and Computer Science}.
  Springer, 2015, pp. 31--47.

\bibitem{thang2016outcome}
T.N. Thang and N.T.B. Kim, \emph{Outcome space algorithm for generalized
  multiplicative problems and optimization over the efficient set}, JIMO 12
  (2016), pp. 1417--1433.

\bibitem{thang2016solving}
T.N. Thang, D.T. Luc, and N.T.B. Kim, \emph{Solving generalized convex
  multiobjective programming problems by a normal direction method},
  Optimization 65 (2016), pp. 2269--2292.

\bibitem{Thang2020}
T.N. Thang, V.K. Solanki, T.A. Dao, N. Thi Ngoc~Anh, and P. Van~Hai, \emph{A
  monotonic optimization approach for solving strictly quasiconvex
  multiobjective programming problems}, Journal of Intelligent {\&} Fuzzy
  Systems 38 (2020), pp. 6053--6063. 5.

\bibitem{tran2023framework}
T.A. Tran, L.P. Hoang, D.D. Le, and T.N. Tran, \emph{A framework for
  controllable pareto front learning with completed scalarization functions and
  its applications}, arXiv preprint arXiv:2302.12487  (2023).

\bibitem{Tuy2016}
H. Tuy, \emph{Convex Analysis and Global Optimization}, Springer Cham, 2016.

\bibitem{Ukkusuri2013}
S. Ukkusuri, K. Doan, and H.A. Aziz, \emph{A bi-level formulation for the
  combined dynamic equilibrium based traffic signal control}, Procedia - Social
  and Behavioral Sciences 80 (2013), pp. 729--752. 20th International Symposium
  on Transportation and Traffic Theory (ISTTT 2013).

\bibitem{White1993}
D.J. White and G. Anandalingam, \emph{A penalty function approach for solving
  bi-level linear programs}, Journal of Global Optimization 3 (1993), pp.
  397--419.

\bibitem{Yu2019}
J. Yu, A. Vexler, and K. Jalal, \emph{A critical issue of using the variance of
  the total in the linearization method - in the context of unequal probability
  sampling}, Statistics in Medicine 38 (2019), pp. 1475--1483.

\end{thebibliography}

\end{document}